\documentclass[a4paper,11pt]{elsarticle}
\usepackage{amsthm}
\usepackage{amssymb}
\usepackage{enumerate}
\usepackage{tikz}
\usepackage{times}

\newcommand\junk[1]{}

\makeatletter
\def\ps@pprintTitle{%
  \let\@oddhead\@empty
  \let\@evenhead\@empty
  \def\@oddfoot{\reset@font\hfil\thepage\hfil}
  \let\@evenfoot\@oddfoot
}
\makeatother
\newtheorem{theorem}{Theorem}[section]

\newtheorem{example}[theorem]{Example}
\newtheorem{lemma}[theorem]{Lemma}

\def\Ds{\mathcal{D}}
\def\As{\mathcal{A}}

\begin{document}
\begin{frontmatter}
\title{On infinite-finite duality pairs of directed graphs}
\author[renyi]{P\'eter L. Erd\H os\fnref{elp}}
\author[canada]{Claude Tardif \fnref{tardif}}
\author[renyi]{G\'abor Tardos\fnref{tardos}}
\address[renyi]{Alfr\'ed R{\'e}nyi Institute of Mathematics, Re\'altanoda u 13-15 Budapest,
        1053 Hungary\\
        {\tt email}: $<$elp,tardos$>$@renyi.hu}
\address[canada]{Royal Military College of Canada, PO Box 17000 Station ``Forces'' \\
        Kingston, Ontario, Canada, K7K 7B4\\
        {\tt email}: Claude.Tardif@rmc.ca}
\fntext[elp]{Research supported in part by the Hungarian NSF, under contract NK 78439
    and  K 68262}
\fntext[tardif]{Research supported by grants from NSERC and ARP.}
\fntext[tardos]{Research supported in part by the NSERC grant 329527 and by the Hungarian NSF grants T-046234, AT048826 and NK-62321}

\begin{abstract}
The $(\As,\Ds)$ duality pairs play crucial role in the  theory of general
relational structures and in the Constraint Satisfaction Problem. The case
where both classes are finite is fully characterized. The case when both side
are infinite seems to be very complex. It is also known that no
finite-infinite duality pair is possible if we make the additional restriction
that both classes are antichains. In this paper (which is the first one of a
series) we start the detailed study of the infinite-finite case.

Here we concentrate on directed graphs. We prove some elementary properties of the infinite-finite duality pairs, including lower and upper bounds on the size of $\Ds$, and show that the elements of $\As$ must be equivalent to forests if $\As$ is an antichain. Then we construct instructive examples, where the elements of $\As$ are paths or trees. Note that the existence of infinite-finite antichain dualities was not previously known.
\end{abstract}
\end{frontmatter}

\section{Introduction}
\noindent
In this paper we consider {\em directed graphs}, {\em homomorphisms} between
them, and especially {\em duality pairs}. We start with the definitions.

A directed graph $G$ is a pair $(V,E)$ with $V=V(G)$ the set of vertices and $E=E(G)\subseteq V^2$ the set of (directed) edges. Unless stated otherwise ``graph'' refers to finite directed graphs in this paper. Forgetting about the orientation of the edges one gets the {\em underlying undirected graph}. For simplicity we drop the term ``oriented'' when referring to (oriented) {\em paths}, (oriented) {\em trees} and (oriented) {\em forests}, these are (directed) graphs whose underlying undirected graphs are path, trees, respectively forests in the traditional sense. In particular, paths, trees and forests have no loops and no pair of vertices is connected in both directions. Similarly, when we call a graph {\em connected}, refer to the {\em connected components} or the {\em girth} of a graph or to a {\em cycle} in a graph we mean the corresponding notion in the underlying undirected graphs.

A {\em homomorphism} $f$ between graphs $G$ and $H$ is a map
$f:V(G)\to V(H)$ satisfying that for every edge $(x,y)\in E(G)$ we have
$(f(x),f(y))\in E(H)$. We write $f:G\to H$ to express that $f$ is a
homomorphism from $G$ to $H$ and we write $G\to H$ to express that such a
homomorphism exists. This is clearly a transitive and reflexive relation. We
write $G\not\to H$ if no homomorphism from $G$ to $H$ exists and call a family
of graphs an {\em antichain} if no homomorphism exists between any two
distinct members.

If both $G\to H$ and $H\to G$ hold for a pair of graphs we say $G$
and $H$ are {\em equivalent}. This is
clearly an equivalence relation. In any equivalence class the graph with the
fewest vertices is unique up to isomorphism. We call such a graph a
{\em core} and also the {\em core of} any graph in its equivalence
class. A graph $G$ is a core if and only if every homomorphism $f:G\to G$ is
an isomorphism.

We
say a graph $G$ is {\em minimal} in a family $\As$ of graphs, if $G\in\As$ and
any graph $H\in\As$ satisfying $H\to G$ is equivalent to $G$. We define
{\em maximal} in a family of graphs similarly but with the homomorphism
condition reversed. Note that there are two-way infinite chains of graphs, so
infinite classes do not always have minimal or maximal elements.

A {\em duality pair} is a pair $(\As,\Ds)$ of families of graphs satisfying
that for every graph $G$ we have either $A\to G$ for some $A\in\As$ or $G\to
D$ for some $D\in\Ds$ but not both. If $(\As,\Ds)$ is a duality pair we call
$\Ds$ a {\em dual} of $\As$. Note, however, that this relation is not
symmetric.

Clearly, each graph in $\As$ and $\Ds$ can be replaced with its core to obtain another duality pair $(\As',\Ds')$ so we can (and often will) assume that both sides of a duality pair consist of cores. Further if $A\to A'$ with $A\ne A'$ and $A,A'\in\As$ we can remove $A'$
from $\As$ without ruining the duality pair property. This way, if $\As$ is {\em finite} we can replace it with the set $\As'$ of its minimal elements and in the resulting duality pair $(\As',\Ds)$ where $\As'$ is an antichain. Similarly, if $\Ds$ is finite we can replace it with the set $\Ds'$ of its maximal elements to obtain a duality pair $(\As,\Ds')$ with $\Ds'$ being an antichain. Note however that such transformation is not possible in general for infinite families.

It is a trivial observation that any family $\As$ has a dual set $\Ds$, simply take $\Ds=\{G\mid\not\hspace{-4pt}\exists A\in\As:A\to G\}$. For any family $\Ds$ of graphs one  can similarly set $\As=\{G\mid\not\hspace{-3.5pt}\exists D\in\Ds:G\to D\}$ making $(\As, \Ds)$ a duality pair. Because of this abundance it is not reasonable to hope for a meaningful characterization of all duality pairs. But characterization of restricted classes of duality pairs have been already done successfully. By "unique" in the following results we mean unique up to equivalence (or up to isomorphism if we restrict attention to cores).

\begin{theorem}[\cite{NT}]
Each tree $t$ has a well-defined, unique graph $D(t)$ making $(\{t\},\{D(t)\})$ a duality pair. In all singleton duality pairs $(\{A\},\{D\})$ the graph $A$ is equivalent to a tree.
\end{theorem}

\begin{theorem}[\cite{FNT}]\label{th:finiteforest}
All finite families $\As$ of forests have a unique antichain dual $\Ds(\As)$ and it is finite. For any duality pair $(\As,\Ds)$ with both $\As$ and $\Ds$ finite antichains all graphs $A\in\As$ are equivalent to forests.
\end{theorem}
We saw above that having antichains as the members of duality pairs can be considered as a relaxation of the finiteness condition. Having characterized the duality pairs with both sides finite it is natural to consider this relaxation. We start with quoting a result showing that there are probably too many infinite-infinite antichain duality pairs for a meaningful
characterization.

\begin{theorem}[\cite{DENS}]
Each finite antichain $\As$ of graphs, that is not maximal, can be extended
\begin{enumerate}[{\rm (i)}]
\item to a duality pair $(\mathcal B,\mathcal C)$ such   that $\mathcal A \subset \mathcal B$ and both $\mathcal B$ and $\mathcal C$ are infinite antichains;
\item to a maximal infinite antichain, which is not a union of the sides of   any duality pair.
\end{enumerate}
\end{theorem}
\noindent This leaves open the question of finding or characterizing duality pairs with
one side finite while the other an infinite antichain. Erd\H{o}s and Soukup
\cite{ES} proved that no such duality pair exists with the left side finite.

\begin{theorem}[\cite{ES}]
There exists no duality pair $(\As,\Ds)$ with $\As$ finite and $\Ds$ an infinite antichain.
\end{theorem}

The existence of
infinite-finite antichain duality pairs has been a long standing open problem. In this paper we give several examples of such duality pairs and also study what families can appear in the left side of such a duality pair. The final answer (a characterization of such families) will follow from the upcoming paper \cite{general} that studies the problem in the more general
context of relational structures.

In Section~\ref{sec:why} we limit the complexity of any graph appearing in an antichain with a finite dual: it must be equivalent to a forest. For finite antichains this is implied by Theorem~\ref{th:finiteforest}. We also show that such a family has to have bounded maximal degree and bounded number of components.

When our forests in a duality pair have only one component and maximum degree two we deal with families of paths. In Section~\ref{sec:path} we exhibit specific infinite antichains of paths, some with and some without a finite dual.

In Section~\ref{sec:tree} we give a simple transformation turning the duality pairs in the previous section into ones with non-path trees on the left side but these trees are still close to paths. We also give an examples of infinite-singleton duality pair where the left side consists of more complex trees constructed from arbitrary binary trees. One of these examples has an antichain on the left side. In another example of an infinite-finite antichain duality pair the left side consists of forests with several components.

\section{Why forests?}\label{sec:why}
\noindent
In this section we prove that all graphs in the left side of an infinite-finite antichain duality pair must be equivalent to forests. This is an extension of the corresponding result for finite antichains in Theorem~\ref{th:finiteforest}.

\begin{theorem}\label{th:nocycle-antichain}
Let $(\As,\Ds)$ be a duality pair, where $\Ds$ is finite and $\As$ consists of cores. Then for each graph $A\in\As$ that is not a forest there exists another graph $B\in\As$ with $B\to A$ but $A\not\to B$.
\end{theorem}
\noindent This result can be proved from the Directed Sparse Incomparability Lemma, see
\cite{DENS,ES}. We present a self contained proof instead.

\begin{proof}
Let $A\in\As$ be a graph that is no forest. Let $(x,y)$ be an edge of $A$ contained in a cycle $C$. Let $A'$ be the graph obtained from $A$ by removing this edge, adding a new vertex $x'$ and the edge $(x',y)$. Notice that the map moving $x'$ to $x$ and fixing all other vertices is an $A'\to A$ homomorphism.

Let  $X$ be an arbitrary tournament with more vertices than any of the (finitely many) graphs in $\Ds$. Let us consider the vertex set $V(X)$ (no edges yet) and disjoint copies  $A'_{uv}$ of $A'$ for every edge $(u,v)\in E(X)$. We obtain the graph $Y$ by identifying the copy of $x$ in $A'_{uv}$ by $u$ and the copy of $x'$ in $A'_{uv}$ by $v$ for all $(u,v)\in E(X)$.

Note that the natural $A'\to A$ homomorphism can be applied on each copy $A'_{uv}$ of $A'$ as all the identified vertices are mapped to $x$. This gives us a natural homomorphism $g:Y\to A$.

As $(\As,\Ds)$ is a duality pair we either have a graph $B\in\As$ with $B\to Y$ or a graph $D\in\Ds$ with $Y\to D$. In the latter case we have $|V(X)|>|V(D)|$, so by the pigeonhole principle we must have $f(u)=f(v)$ for an edge $(u,v)\in E(X)$. But this means that $f$ restricted to $A'_{uv}$ is an $A\to D$ homomorphism, a clear contradiction. This leaves the former possibility only. We show that $B\in\As$ with $B\to Y$ satisfies the statement of the theorem.

Indeed we have $B\to Y\to A$. We will show $A\not\to Y$ and this implies $A\not\to B$. In the degenerate case when $A$ consist of a single loop edge $A\not\to Y$ holds, since $Y$ is a tournament in this case. So we may assume $A$ is not a loop and as it is a core it does not even contain a loop. In particular $x\ne y$ and $C$ has length at least 2. Assume for a contradiction that a homomorphism $f:A\to Y$ exists. As $A$ is a core the homomorphism $f\circ g:A\to A$ must be an automorphism. Modifying $f$ appropriately, one can assume without loss of generality that $f\circ g$ is the identity, so $f(z)\in g^{-1}(z)$ for each vertex $z\in V(A)$. We must have $f(x)\in g^{-1}(x)=V(X)$ and $f(z)\notin V(X)$ for any other vertex $z$ of $A$. The vertices of the cycle $C$ except $x$ itself must be mapped in a single connected component of $Y\setminus V(X)$, in particular, in a single copy $A'_{uv}$ of $A'$. The image $f(y)$ of $y$ must be the copy of $y$ in $A'_{uv}$, so to have $(f(x),f(y))\in E(Y)$ we must have $f(x)=v$. This forces the image of the other edge incident to $x$ in the cycle $C$ outside $E(Y)$. The contradiction finishes the proof of the theorem.
\end{proof}

Note that this theorem implies that if we have antichains in an infinite-finite duality pair of cores, then the left side contains forests only. Furthermore, something can be said without restricting attention to antichains. Let $(\As,\Ds)$ be an infinite-finite duality pair of cores. We can remove from $\As$ all the graphs which are not forests but are ``dominated'' by one: the graphs $A\in\As$ for which a forest $B\in\As$ exists with $B\to A$, but $A$ itself is not an
forest. Clearly, the set of graphs to where a homomorphism exists from a member of the remaining family $\As'$ did not change, so $(\As',\Ds)$ is still a duality pair. This duality pair may still contain a graph $A\in\As'$ that is not a forest, but such a graph must have {\em infinitely many} distinct graphs $B\in\As'$ with $B\to A$. Indeed, if $A$ has only finitely many such dominating $B$, then any minimal graph in this finite set would violate the preceding theorem. From the Directed Sparse Incomparability Lemma one can also show that if a graph $A\in\As'$ is {\em imbalanced} (containing a cycle with an unequal number of forward and reverse oriented edges), then the underlying undirected graphs of the graphs $B\in\As'$ have unbounded girth.

In the following lemma we state the connection between having connected graphs on the left side of a duality pair and having a single graph on the right side. Recall that we call a graph {\em connected} if the underlying undirected graph is connected and use the term {\em connected component} in a similar way.

\begin{lemma}\label{lem:connected}
Let $(\As,\Ds)$ be an antichain duality pair with $\As$ consisting of cores. If a graph $A\in \As$ has $k$ connected components we have $|\Ds|\ge k$. But if all graphs in $\As$ are connected, then $|\Ds|=1$.
\end{lemma}

\begin{proof}
Take a graph $A\in\As$. Let $A_1,\ldots,A_k$ be the graphs obtained from $A$ by removing a single one of its $k$ components. For any $1\le i\le k$ we have $A\not\to A_i$ (since  $A$ is a core) furthermore we have $A_i\to A$, so as $\As$ is an antichain it contains no graph that has a homomorphism to $A_i$. As $(\As,\Ds)$ is a duality pair each $A_i$ has a graph $D_i\in\Ds$ with $A_i\to D_i$. If we have $D_i=D_j$ for some $1\le i<j\le k$ we can construct an $A\to D_i$ homomorphism by extending the $A_i\to D_i$ homomorphism to the missing component using the corresponding restriction of the $A_j\to D_i$ homomorphism. Since we must have $A\not\to D$ for $D\in\Ds$ all graphs $D_i$ are distinct and thus $|\Ds|\ge k$ as claimed.

Now assume that all $A\in\As$ is connected but still we have two graphs $D_1 \ne D_2$ in $\Ds$. As $\Ds$ is an antichain the disjoint union $D$ of $D_1$ and $D_2$ does not have a homomorphism to any member of $\Ds$. By the duality pair property we must have a graph $A\in \As$ such that $A\to D$. As $A$ is connected this homomorphism maps $A$ either to $D_1$ or to $D_2$, giving $A\to D_1$ or $A\to D_2$, a contradiction.
\end{proof}

An immediate corollary of this lemma is that if an antichain has a finite dual its members have a bounded number of components. For this we do not need the full strength of the antichain condition, it is enough to assume that we do not have homomorphism between two graphs of $\As$ that avoids an entire connected component of the target graph. While replacing the left hand side of a duality pair with an equivalent antichain is not always possible, it is easy to see that replacing the left hand side of a duality pair with an equivalent family satisfying this constraint is always possible. If the right side is finite, then after this transformation the graphs in the left side have a bounded number of components.

We end this section by showing that the maximum degree is also bounded in an antichain of core graphs that has a finite dual.

\begin{lemma}\label{lem:degree}
Let $(\As,\Ds)$ be a duality pair with $\As$ an antichain consisting of cores and $\Ds$ finite. Any vertex of any graph $A\in\As$ has total degree $($this is the sum of the in-degree and out-degree$)$ at most $d_0=\sum_{D\in\Ds}|V(D)|$.
\end{lemma}

\begin{proof}
Let $A\in\As$ and $v\in V(A)$ and suppose the total degree  $d$ of $v$ is larger than $d_0$. By Theorem~\ref{th:nocycle-antichain} $A$ is  a forest, so $v$ cuts its component of $A$ into $d$ parts. Let us form the subgraphs $A_1,\ldots,A_d$ of $A$ by removing a single one of these parts from $A$. That is, each $A_i$ is obtained from $A$ by removing an edge $e$
from $A$ that connects $v$ to another vertex $w$ (with either orientation) and also removing the connected component of $w$ from the resulting graph. As $\As$ is an antichain of cores no member of $\As$ has a homomorphism to any of these subgraphs $A_i$, so by the duality pair property, there must be homomorphisms $f_i:A_i\to D_i$ from $A_i$ to certain graphs $D_i\in D$. From $d>d_0$ we must have $1\le i<j\le d$ with $D_i=D_j$ and $f_i(v) =f_j(v)$. We construct an $A\to D_i$ homomorphism by extending $f_i$ with the restriction of $f_j$ to the part of $A$ missing from $A_i$. The contradiction finishes the proof.
\end{proof}

\section{Antichains of paths}\label{sec:path}
\noindent
In this section we give concrete examples of infinite antichains of paths with or without a finite dual. When looking for a finite dual it is always enough to consider duals consisting of a single graph by Lemma~\ref{lem:connected}.

To speak of (oriented) paths we use the natural correspondence between them and words over the binary alphabet $\{+,-\}$. We use standard notation with respect to these words, namely a word is a member of $\{+,-\}^*=\cup_{k\ge0}\{+,-\}^k$, where $\{+,-\}^k$ is the set of length $k$ sequences from the alphabet. For $x,y\in\{+,-\}^*$ and $k\ge0$ we write $xy$
for the concatenation of $x$ and $y$ and $x^k$ for the word obtained by concatenating $k$ copies of $x$.

The correspondence is given by the map $p$ mapping $\{+,-\}^*$ to paths. For a word $x=x_1\ldots x_k\in\{+,-\}^k$ let $p(x)$ stand for the path consisting of $k$ edges with the $i$'th edge oriented forward if $x_i=+$ and backward otherwise. A bit informally we will refer to the first and last vertices of $p(x)$ in their obvious meaning, although formally the end vertices of the path $p(x)$ cannot be distinguished without knowing $x$. We say that a
homomorphism $f:p(x)\to G$ {\em maps $p(x)$ from $u$ to $v$} if the image of the first vertex of $p(x)$ is $u\in V(G)$ and the image of the last vertex is $v\in V(G)$. Note that although all (isomorphism classes of) paths will be obtained as images in this map the
correspondence is not one-one: up to two distinct words may be mapped to isomorphic paths (with the role of the first and last vertices reversed), for example $p(++-)$ and $p(+--)$ are isomorphic.

\subsection{Antichains without a final dual}

A trivial observation is the following: Take any infinite antichain of paths (as we will see such antichains are easy to find). The cardinality of the set of its subsets is continuum, and no two can have the same set for dual. Thus many have no finite dual, as the set of finite families of graphs is countable. This cardinality argument gives no explicit family without a finite dual. Here we set out to construct such a set.

\begin{lemma}\label{lem:power}
Let $G$ be a graph with $|V(G)| \le k$ and assume that for some $x\in\{+,-\}^*$ we have $p(x^k)\rightarrow G$. Then for each $\ell\ge0$ we also have $p(x^\ell)\rightarrow G$.
\end{lemma}
\begin{proof}
By considering the homomorphism $p(x^k)\to G$ one finds vertices $v_0,\ldots, v_k$ in $G$ such that for each $1\le i\le k$ a suitable restriction of the homomorphism maps $p(x)$ from $v_{i-1}$ to $v_i$. By the pigeon hole principle we find $v_i=v_j$ for some $0\le i<j\le k$. Thus, we can map the $p(x^{j-i})$ to $G$ with both endpoints mapping to the same vertex. This closed walk can take the homomorphic image of $p(x^\ell)$ for any $\ell$.
\end{proof}

\begin{example}
Let $Q_k=p((+(+-)^k)^k++)$ and consider any infinite family $\As\subseteq\{Q_k\mid k\ge1\}$. Then $\As$ is an antichain of paths and has no finite dual.
\end{example}
\begin{proof}
To see that $\As$ is an antichain observe that the {\em hight} of a path defined as the maximal difference between forward and backward edges in a sub-path cannot be decreased by a homomorphism. As the hight of $Q_k$ is $k+2$ we have $Q_k\not\to Q_\ell$ for $\ell<k$. But $p(+(+-)^k++)$ is a sub-path of $Q_k$ and even this sub-path does not map to $Q_\ell$ for $\ell>k$. A similar argument also shows that all $Q_k$ are cores: as deleting either the first or the last edge of $Q_k$ decreases its hight any homomorphism $Q_k\to Q_k$ must be onto and thus an isomorphism.

Assume $(\As,\Ds)$ is a duality pair. Let $Q_k\in\As$ and consider $Q_k'=p((+(+-)^k)^k)$. Clearly, $Q_k'\to Q_k$, so we have $Q_\ell\not\to Q_k'$ for $\ell\ne k$ by the antichain property and $Q_k\not\to Q_k'$ since $Q_k$ is a core. So we must have $D\in\Ds$ with $Q_k'\to D$. We claim that $|V(D)|>k$. Indeed, otherwise by Lemma~\ref{lem:power} we have $Q_k\to p((+(+-)^k)^{k+1})\to D$, a contradiction. As $\As$ is infinite $k$ could be chosen arbitrarily large, so $\Ds$ must have arbitrarily large graphs, it cannot be finite.
\end{proof}

\subsection{Infinite antichains of paths with a finite dual}

Our first infinite-finite antichain duality pair, the $s=3$ case of the next example, is the smallest possible such example in the sense that the dual is a single graph on four vertices, while no graph or family of graphs on fewer vertices is a dual of an infinite antichain.

\begin{example}\label{ex:pd}
Let $P^s_k=p(+^s(-+^{s-1})^k+)$ for $s\ge1$, $k\ge0$ and let $D_s$ be the  graph obtained from the transitive tournament on $s+1$ vertices by deleting the edge connecting the source and the sink. Then $(\{P^s_k\mid k\ge0\},\{D_s\})$ is an antichain duality pair of cores for
$s\ge3$.
\end{example}

\begin{proof}
To see that the infinite side is an antichain of cores we partition $P^s_k$ into $k+2$ parts, the first being the directed path $p(+^s)$, the next $k$ parts being $p(-+^{s-1})$, the last part being a single edge. In any $P^s_k\to P^s_\ell$ homomorphism the first part of $P_k^s$ must not map to last $s$ edges of $P^s_\ell$ because that would make the mapping of the next part impossible. So it must be mapped identically to the first part of $P^s_\ell$
and then the next $k$ parts of $P^s_k$ must also map identically to the next $k$ part of $P^s_\ell$. This only works if $\ell\ge k$. But if $\ell>k$ the last edge of $P^s_k$ cannot be mapped anywhere. So we must have $k=\ell$ and the homomorphism must be the identity.

To see that $D_s$ is also a core it is enough to note that it is acyclic and has a directed Hamiltonian path. Let us denote the vertices along this path by $v_0,\ldots,v_s$.

We show $P^s_k\not\to D_s$ similarly to the antichain property. Indeed, the first part of $P^s_k$ (forming a directed path) has a single homomorphism to $D_s$ ending at $v_s$. Each of the next $k$ parts must map to the path $v_sv_1v_2\cdots v_s$. But as $v_s$ is a sink, this homomorphism cannot be extended to the last edge of $P^s_k$.

Let $G$ be an arbitrary graph. By the statement in the last paragraph we cannot have $P^s_k\to G\to D_s$ for any $k\ge0$. So it remains to prove that either $P^s_k\to G$ for some $k\ge0$ or we have $G\to D_s$.

We call a vertex $v\in V(G)$ {\em type} $i$ for $0\le i\le s$ if it is the image under a homomorphism of the last vertex of a path $p(+(+^{s-1}-)^k+^i)$ for some $k\ge0$. Note that for $i\ge1$ a type $i$ vertex is the image of the last vertex of the path $p(+^i)$ so it is also a type $i-1$ vertex.

If there is a type $s$ vertex in $G$ we clearly have $P^s_k\to G$ for some $k\ge0$ and we are done.

If there is no type $s$ vertex in $G$ we define $\phi:V(G)\to V(D_s)$ by setting $\phi(v)=v_0$ if $v$ is not type 0 and for $1\le i\le s$ setting $\phi(v)=v_i$ if $v$ is not type $i$ but $v$ is type $i-1$.

We claim that $\phi$ is a $G\to D_s$ homomorphism.

Let $(u,v)$ be an edge of $G$. This makes $v$ the endpoint of an edge, so it is type 0. Moreover, if $u$ is type $i$, then the path $p(+(+^{s-1}-)^k+^i)$ mapping to $G$ and ending at $u$ can be extended by the $(u,v)$ edge, making $v$ a type $i+1$ vertex. Thus if $\phi(u)=v_j$ and $\phi(v)=v_{j'}$ we must have $j<j'$. It remains to prove that $\phi(u)=v_0$ and $\phi(v)=v_s$ is impossible. Indeed, $\phi(v)=v_s$ implies $v$ is type $s-1$, so it is the image of the last vertex of a path $p(+(+^{s-1}-)^k+^{s-1})$. Extending this with the $(u,v)$ edge we get that $u$ is the image of the last vertex of the path $p(+(+^{s-1}-)^{k+1})$ making $u$ a type 0 vertex. This finishes the proof.
\end{proof}

\subsection{Regularity}

From the two examples considered so far one can notice that relevance of regular languages. Indeed, while the family of words $\{+^s(-+^{s-1})^k+\mid k\ge0\}$ is a regular language for any $s$, the family $\{+(+-)^k)^k++\mid k\ge1\}$ or any of its infinite subfamilies are not regular. This connection was the basis of our upcoming paper \cite{caterpillars} that
establishes regularity as necessary and sufficient condition for having a finite dual in this case. We state the following easy observation regarding regularity to further motivate this connection.

\begin{lemma}
Let $G$ be an arbitrary graph. The set $\{x\in\{+,-\}^*\mid p(x)\to G\}$ is a regular language.
\end{lemma}

\begin{proof}
We turn the graph $G$ into a nondeterministic finite automaton. The states of the automaton are the vertices of $G$ and each state is an initial and also a terminal state. For each edge $(u,v)$ of $G$ we make the transition from $u$ to $v$ possible for the letter $+$ and the transition from $v$ to $u$ possible for the letter $-$. It is straightforward to see that this automaton accepts the language in the lemma.
\end{proof}

\section{Antichains of trees}\label{sec:tree}
\noindent
In this section we give infinite-finite antichain duality pairs where the infinite side has trees that are not paths. The following lemma is instructive for this.

\begin{lemma}\label{lem:append}
Let $(\As,\Ds)$ be a duality pair. Let us modify each $A\in\As$ by enriching  it with new  vertices and edges:  from each sink of $A$ we start a new edge to a separate new vertex . Let $\As'$ be the family of these modified graphs. Let us modify each graph $D\in\Ds$ by adding a single new vertex and edges to this vertex from every vertex of $D$. Let $\Ds'$ be the family of these modified graphs. Then $(\As',\Ds')$ is a duality pair. If $\As$ is an antichain so is $\As'$, if $\Ds$ is an antichain, so is $\Ds'$, furthermore if $\As$ and $\Ds$ consists of cores so do $\As'$ and $\Ds'$.
\end{lemma}
\begin{proof}
Let $A'\in\As'$ be the modification of $A\in\As$ and $D'\in\Ds'$ be the modification of $D\in\Ds$. If we have a homomorphism $f:A'\to D'$, then its restriction to $A$ must map to $D$ as the single vertex of $D'\setminus D$ is a sink in $D'$, but no vertex of $A$ is a sink in $A'$. But the existence of an $A\to D$ homomorphism contradicts the fact that $(\As,\Ds)$ is a duality pair.

Let $G'$ be an arbitrary graph. We cannot have $A'\to G'\to D'$ for some $A'\in\As'$ and $D'\in\Ds'$ by the previous paragraph. It remains to show that $A'\to G'$ for some $A'\in\As'$ or $G'\to D'$ for some $D'\in\Ds'$.

Let $G$ be the subgraph of $G'$ induced by the non-sink vertices. As $(\As,\Ds)$ is a duality pair we either have $A\to G$ for some $A\in\As$ or $G\to D$ for some $D\in\Ds$. In the former case we can extend the homomorphism $A\to G$ to a homomorphism $A'\to G'$, where $A'\in\As'$ is the modified version of $A$. In the latter case we can extend the
homomorphism $G\to D$ to a homomorphism $G'\to D'$, where $D'\in\Ds'$ is the modified version of $D$, by sending all vertices of $G'\setminus G$ to the single vertex in $D'\setminus D$.

To see that the antichain and core properties are inherited to the modified sets consider two graphs $X$ and $Y$ from the same family $\As$ or $\Ds$ and their modifications $X'$ and $Y'$. Restricting a homomorphism $X'\to Y'$ to $X$ we get a homomorphism $X\to Y$. Indeed, all vertices in $Y'\setminus Y$ are sinks and no vertex in $X$ is sink in $X'$. So if the family was antichain, then $X=Y$ and so the modified family is also an antichain. If $X=Y$ is a core, then the $X\to Y$ homomorphism must be an isomorphism and it is easy to see that the original $X'\to Y'$ homomorphism must also be an isomorphism.
\end{proof}
\noindent Applying this lemma (possibly several times) for our earlier examples of infinite-finite antichain duality pairs we get several new such examples. Although the graphs on the left side of these pairs are no longer paths, they are still very similar to paths in structure.

\bigskip\noindent The examples in the next lemma show better the complexity that families with a finite dual can exhibit.

Let us consider the family $T_0$ of all finite rooted (undirected) binary trees satisfying that each vertex is either a leaf (no children) or it has two children: a left child and a right child. Note that the smallest member of $T_0$ has a single vertex.

Let $x,y,s,z\in\{+,-\}^*$ be words. We define the family of oriented trees $T(x,y,$ $s,z) = \{t(x,y,s,z)\mid t\in T_0\}$, where $t(x,y,s,z)$ is an oriented tree obtained from $t$ by
\begin{enumerate}[{\rm (A)}]
\item replacing each edge connecting a vertex $u$ to its left child $v$ by a   copy of $p(x)$ from $u$ to $v$,
\item replacing each edge connecting a vertex $u$ to its right child $w$ by a   copy of $p(y)$ from $u$ to $w$,
\item adding a path $p(s)$ from each leaf vertex of $t$ and \item adding a path $p(z)$ from the root of $t$.
\end{enumerate}
Let $G_1$ and $G_2$ be the graphs depicted on Figures 1 and 2, furthermore let
$$
T_1=T(+-,-+,--,++) \quad\hbox{and}\quad  T_2=T(+--,-+-,--,+++).
$$

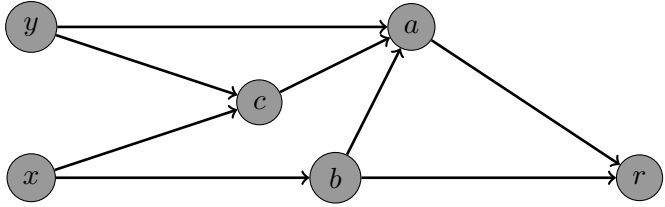
\begin{figure}[h]\label{f:g1}

\begin{tikzpicture}
\node  [shape=circle,draw,fill=black!40] (x) at ( 1,1) {$x$};
\node  [shape=circle,draw,fill=black!40] (y) at ( 1,3) {$y$};
\node  [shape=circle,draw,fill=black!40] (b) at ( 5,1) {$b$};
\node  [shape=circle,draw,fill=black!40] (c) at ( 4,2) {$c$};
\node  [shape=circle,draw,fill=black!40] (a) at ( 6,3) {$a$};
\node  [shape=circle,draw,fill=black!40] (r) at ( 9,1) {$r$};
\draw [line width=1pt,->] (x) -- (b);
\draw [line width=1pt,->] (x) -- (c);
\draw [line width=1pt,->] (y) -- (c);
\draw [line width=1pt,->] (y) -- (a);
\draw [line width=1pt,->] (b) -- (a);
\draw [line width=1pt,->] (b) -- (r);
\draw [line width=1pt,->] (c) -- (a);
\draw [line width=1pt,->] (a) -- (r);

\end{tikzpicture}
\caption{The graph $G_1$.}
\end{figure}

\begin{figure}[h]\label{f:g2}
\begin{tikzpicture}
\node  [shape=circle,draw,fill=black!40] (x) at ( 1,1) {$x$};
\node  [shape=circle,draw,fill=black!40] (y) at ( 1,5) {$y$};
\node  [shape=circle,draw,fill=black!40] (b) at ( 5,1) {$b$};
\node  [shape=circle,draw,fill=black!40] (c) at ( 4,3) {$c$};
\node  [shape=circle,draw,fill=black!40] (a) at ( 5.7,3) {$a$};
\node  [shape=circle,draw,fill=black!40] (r) at ( 7.4,4) {$r$};
\node  [shape=circle,draw,fill=black!40] (s) at ( 12,4.5) {$s$};
\draw [line width=1pt,->] (x) -- (b);
\draw [line width=1pt,->] (x) -- (a);
\draw [line width=1pt,->] (x) -- (c);
\draw [line width=1pt,->] (y) -- (c);
\draw [line width=1pt,->] (y) -- (r);
\draw [line width=1pt,->] (y) -- (s);
\draw [line width=1pt,->] (y) -- (a);
\draw [line width=1pt,->] (b) -- (a);
\draw [line width=1pt,->] (b) -- (r);
\draw [line width=1pt,->] (b) -- (s);
\draw [line width=1pt,->] (c) -- (a);
\draw [line width=1pt,->] (a) -- (r);
\draw [line width=1pt,->] (a) -- (r);
\draw [line width=1pt,->] (a) -- (s);
\draw [line width=1pt,->] (r) -- (s);
\draw [line width=1pt,->] (c) -- (r);

\end{tikzpicture}
\caption{The graph $G_2$.}
\end{figure}
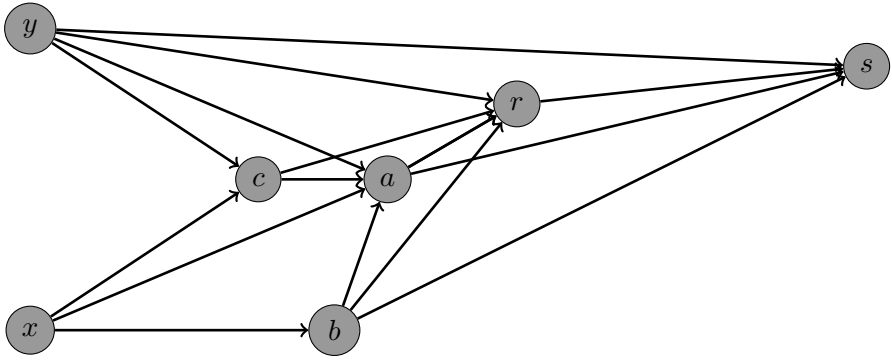

\begin{theorem}\label{th:fas}
\begin{enumerate}[{\rm(i)}]
\item $\big (T_1,\{G_1\}\big )$ is a duality pair of core graphs.
\item $\big (T_2,\{G_2\}\big )$ is an antichain duality pair of core graphs.
\end{enumerate}
\end{theorem}

\begin{proof}
We leave the simple proofs that all involved graphs are cores and that $T_2$  is an antichain to the diligent reader. Note that $T_1$ is not an antichain: if $t\in T_0$ and $t'$ is subtree of $t$ containing the root, then $t(+-,-+,--,++) \to t'(+-,-+,--,++)$. As a result (once (i) is proved) one also has that $(\{t_k(+-,$ $-+,--,++)\mid k\ge100\},\{G_1\})$ is a duality pair if
$t_k\in T_0$ is the depth $k$ full binary tree with $2^k$ leaves.

Assume for a contradiction that $f:t(+-,-+,--,++)\to G_1$. We claim that all vertices of $t$ must map to the vertices $a$ or $r$. This is certainly true for the leaves because of the attached paths $p(--)$ could not map to $G_1$ otherwise. Working with a bottom up induction assume that both the left and the right children of the vertex $u$ map to $a$ or $r$. In this case the paths from $u$ to its children must be mapped to $G_1$ from $f(u)$ to either $a$ or $r$. Then we must have $f(u)=a$ as from no other vertex of $G_1$ does  have both a path $p(+-)$ and a path $p(-+)$ to either $a$ or $r$. So the root vertex must also be mapped to $a$ or $r$ and the contradiction comes from no place in $G_1$ for the path $p(++)$
attached to the root.

For (i) it is left to prove that for any graph $X$ we either have $A\to X$ for an $A\in T_1$ or we have $X\to G_1$. For this we define the ``level'' $L_i\subseteq V(X)$ consists of the vertices of $X$ with a homomorphism $P_i\to X$ ending at $u$ but no homomorphism $P_{i+1}\to X$ ending at $u$. Here $P_i=p(+^i)$ is the directed path of $i$ edges and the levels $L_0$, $L_1$, $L_2$ and $L_3$ partition $V(x)$ or we have $P_4\to X$ and we are done since $P_4\in T_1$.

We construct the map $\phi: V(X)\to V(G_1)$ as follows:
\begin{enumerate}[(1)]
\item Set $\phi(u)=a$ for any vertex $u\in L_1$ that has a $t\in T_0$ and a homomorphism $f:t(+-,-+,--,\epsilon)\to X$ mapping the root of $t$ to $a$. Here $\epsilon$ stands for the empty word. \item Set $\phi(u)=b$ for any vertex $u\in L_1$ not yet mapped to $a$ that either has an edge $(u,v)$ to a vertex $v\in L_3$ or two edges $(u,v)$ and $(w,v)$ with $w\in L_1$ already mapped to $a$.
\item Set $\phi(u)=c$ for the remaining vertices $u\in L_1$.
\item Set $\phi(u)=x$ if $u\in L_0$ and there exists an edge $(u,v)$ with   $\phi(u)=b$.
\item Set $\phi(u)=y$ for the remaining vertices $u\in L_0$.
\item Set $\phi(u)=a$ if $u\in L_2$ and there is no $v\in L_1$ with   $\phi(u)=a$ and $(v,u)$ an edge.
\item Set $\phi(u)=r$ for all remaining vertices in $V(X)$.
\end{enumerate}
\smallskip
\noindent
If $\phi$ is a homomorphism $X\to G_1$ we are done. Otherwise one of the steps above made an edge in $X$ map outside $X$.

Steps 1--3 map the independent set $L_1$ so they caused no problem. In step 4 can create a problem if a vertex $u\in L_0$ has edges $(u,v)$ and $(u,w)$ with $v,w\in L_1$, $\phi(v)=b$ and $\phi(w)=a$. But in this case the homomorphisms triggering $\phi(v)=b$ and $\phi(w)=a$ in steps 2 and 1 can be combined (together with the $wuv$ path) to a homomorphism triggering $\phi(v)=a$ in the first step, a contradiction.

Step 5 cannot cause trouble as both $(y,a)$ and $(y,c)$ are edges in $G_1$.

Steps 6 or 7 cause trouble if there is a vertex $u\in L_2\cup L_3$ with $(v,u)$ an edge from a vertex $v\in L_0$ with $\phi(v)=x$. But then there is a vertex $w\in L_1$ with $(v,w)$ an edge and $\phi(w)=b$. Here again, the homomorphism of $p(++-+)$ to $X$ ending in the vertices $uvw$ can be combined to the homomorphism triggering $\phi(w)=b$ to obtain a homomorphism triggering $\phi(w)=a$, a contradiction.

Finally in step 7 we can map both ends of an $(u,v)$ edge to $r$. This happens if there exists an edge $(w,u)$ from a vertex $w\in L_1$ with $\phi(w)=a$. This may indeed happen, but then the homomorphism triggering $\phi(w)=a$ can be combined to the directed path $wuv$ to get $A\to X$ for a tree $A\in T_1$. This finishes the proof of part (i).

\medskip\noindent (ii) The proof of this part is only slightly more complicated.

Assume for a contradiction that $f:t(+--,-+-,--,+++)\to G_2$. We claim that all vertices of $t$ must map to the vertices $a$, $r$ or $s$. This can be shown exactly like the corresponding statement in part (i). So the root vertex must also be mapped to $a$ or $u$ and the contradiction comes from no place in $G_2$ for the path $p(+++)$ attached to the root.

Finally we assume $X$ is graph with no $A\to X$ homomorphism for any $A\in T_2$. We construct the homomorphism $\phi:X\to G_2$ similarly to part (i).  We partition $V(X)$ into levels $L_i$ as we did above. As $P_5\in T_2$ does not map to $X$, $V(X)$ is partitioned into the sets $L_0$, $L_1$, $L_2$, $L_3$  and $L_4$.
\begin{enumerate}[(1)]
\item Set $\phi(u)=a$ for any vertex $u\in L_1$ that has a $t\in T_0$ and a homomorphism $f:t(+--,-+-,--,\epsilon)\to X$ mapping the root of $t$ to $a$.
\item Set $\phi(u)=b$ for any vertex $u\in L_1$ not yet mapped to $a$ that has either a homomorphism mapping $p(+----)$ to $G_2$ from $u$ or a homomorphism of $p(+--)$ to $G_2$ from $u$ to a vertex $v\in L_1$ with $\phi(v)=a$.
\item Set $\phi(u)=c$ for the remaining vertices $u\in L_1$.
\item Set $\phi(u)=x$ if $u\in L_0$ and there exists an edge $(u,v)$ with   $\phi(u)=b$.
\item Set $\phi(u)=y$ for the remaining vertices $u\in L_0$.
\item Set $\phi(u)=a$ if $u\in L_2$ and there is no $v\in L_1$ with   $\phi(u)=a$ and $(v,u)$ an edge.
\item Set $\phi(u)=r$ for all remaining vertices $u\in L_2$.
\item Also set $\phi(u)=r$ for vertices $u\in L_3$ with no edge $(v,u)$ from a   vertex $v\in L_2$ with $\phi(v)=r$.
\item Set $\phi(u)=s$ for all remaining vertices $u\in V(X)$.
\end{enumerate}
The proof that $\phi$ is indeed a homomorphism is almost identical to the corresponding argument in part (i).
\end{proof}
\medskip\noindent
We finish the paper by giving a simple observation how to combine duality pairs to obtain new pairs with several graphs on the right side. For simplicity we restrict attention to combining two duality pairs with single graphs on the right hand side that are incomparable.

In the following lemma and example $A_1\cup A_2$ denotes the disjoint union of the graphs $A_1$ and $A_2$.

\begin{lemma}\label{lem:combine}
Let $(\As_1,\Ds_1)$ and $(\As_2,\Ds_2)$ be duality pairs and let us   partition $\As_i$ into $\As_i'=\{A\in\As_i\mid\exists B\in\As_{3-i}:B\to A\}$ and   $\As_i''=\As_i\setminus\As_i'$ for $i=1,2$.
\begin{enumerate}[{\rm(i)}]
\item $(\As,\Ds)$ is a duality pair, where $\Ds=\Ds_1\bigcup\Ds_2$   and $\As=\{A_1\cup A_2\mid A_1\in\As_1,A_2\in\As_2\}$.
\item $(\As',\Ds)$ is also a duality pair, where   $\As'=\As_1'\bigcup\As_2'\bigcup\{A_1 \cup A_2\mid A_1\in\As_1'',A_2\in\As_2''\}$.
\item If both $\As_i$ are antichain and $|\Ds_1|=|\Ds_2|=1$, then $\As'$ can be made an antichain   with removing possible duplicates: leaving one member only from each   equivalent pair of graph from $\As_1$ and $\As_2$.
\end{enumerate}
\end{lemma}

\begin{proof} For (i) it is enough to note that $A_1\cup A_2\to B$ if and only if $A_1\to B$ and $A_2\to B$.

For (ii) take $A_2\in\As_2'$ and a graph $A_1\in\As_1$ with $A_1\to A_2$. As $A_1\cup A_2$ is equivalent to $A_2$ we can put $A_2$ into the left side of the duality pair $(\As,\Ds)$. But then all graphs $A_1'\cup A_2$ can be removed from there as $A_2$ maps to these graphs. Doing this for all $A_2\in\As_2'$ and similar changes for the graphs in $\As_1'$ one obtains $\As'$ and (ii) is proved.

To prove (iii) take $A\in\As'$ and consider the sets $S_i(A)=\{B\in\As_i\mid B\to A\}$ for $i=1,2$. For a graph $A=A_1\cup A_2$ with $A_1\in\As_1''$, $A_2\in\As_2''$ we have $S_1(A)=\{A_1\}$ and $S_2(A)=\{A_2\}$ since the graphs in $\As_1\bigcup\As_2$ are connected (Lemma~\ref{lem:connected}), thus they map to $A$ if and only if they map to
$A_1$ or $A_2$. For $A\in\As_i'$ and $i=1$ or $2$ we have $S_i(A)=\{A\}$. Since $A\to A'$ implies $S_i(A)\subseteq S_i(A')$ for $A,A'\in\As'$ and $i=1,2$ the only possibility of such a map with $A\ne A'$ is $A\in\As_i'$ and $A'\in\As_{3-i}'$. From $A\in\As_i'$ we have
$A''\in\As_{3-i}$ with $A''\to A\to A'$. As $\As_{3-i}$ is an antichain we must have $A''=A$ and thus $A$ and $A'$ are equivalent.
\end{proof}

We can apply this lemma to combine any two of the several examples of infinite-finite duality pairs in this paper or even one such example with a simple duality with a single tree on the left hand side. We chose the duality pairs $(\{P_k^4\mid k\ge0\},\{D_4\})$ from Example~\ref{ex:pd} and $(T_2,\{G_2\})$ from Theorem~\ref{th:fas}. Note that $P_0^4$ is the directed path with five edges and it appears in $T_2$ but no homomorphism exist from a member of $T_2$ to some $P_k^4$ with $k\ge1$ or vice versa. Thus from   Lemma~\ref{lem:combine} we get the following

\begin{example}
The following is an antichain duality pair of core graphs:
$$
\big(\{P^4_0\}\cup\big\{P^4_k\cup A\mid k\ge1, A\in  T_2\setminus \{P^4_0\}\big\},\{D_4,G_2\}\big).
$$
\end{example}

\bibliographystyle{plain}

\end{document}